\documentclass[12pt, a4paper]{article}
\usepackage{mathtools,amsthm,amsmath,amsfonts,amssymb,tikz-cd,enumitem, graphicx,mathrsfs,bbm,stmaryrd,xcolor}

\tikzcdset{scale cd/.style={every label/.append style={scale=#1},
    cells={nodes={scale=#1}}}}
\usepackage{authblk}
\usepackage{url}
\usepackage{esvect}
\makeatletter
\newcommand{\address}[1]{\gdef\@address{#1}}
\newcommand{\email}[1]{\gdef\@email{\url{#1}}}
\newcommand{\@endstuff}{\par\vspace{\baselineskip}\noindent\small
\begin{tabular}{@{}l}\@address\\\textit{E-mail address:} \@email\end{tabular}}
\AtEndDocument{\@endstuff}
\makeatother
\usepackage{setspace}
\usepackage[utf8]{inputenc}
\usepackage{CJKutf8}
\counterwithin{equation}{section}
\usepackage[pdfencoding=auto,psdextra,colorlinks=true,linkcolor=blue,citecolor=blue,urlcolor = orange]{hyperref}
\usepackage[margin = 1.2in]{geometry}
\usepackage{kantlipsum}
\allowdisplaybreaks
\newtheorem{theorem}{Theorem}[section]

\newtheorem{remark}[theorem]{Remark}
\newtheorem{proposition}[theorem]{Proposition}

\newtheorem{assumption}[theorem]{Assumption}
\newcommand{\citep}[1]{\cite{#1}}

\newcommand{\professor}{\text{Prof.\ }\hspace{-0.03125mm}}
\newcommand{\doctor}{\text{Dr.\ }\hspace{-0.03125mm}}
\newcommand{\edmp}{\text{\normalfont End}}

\begin{document}
\title{\textbf{A Serre type vanishing property of the twisted primitive cohomology}}
\author{Hao Zhuang}
\address{Beijing International Center for Mathematical Research, Peking University, Beijing, China}
\email{hzhuang@pku.edu.cn}
\date{\today}
\maketitle
\begin{abstract}
We prove a Serre type vanishing property for the twisted primitive cohomology of a symplectic manifold. It is based on Tseng and Zhou's vanishing property under the symplectic flatness. These vanishing properties emphasizes the necessity of the symplectic flatness when generalizing certain results from the sheaf cohomology in complex geometry to the primitive cohomology in symplectic geometry. 
\end{abstract} 
\tableofcontents
\section{Introduction}
This work is motivated by Serre's vanishing theorem \cite[Proposition 5.2.7]{huybrechts2005complex} in complex geometry. Suppose that $M$ is a compact K\"ahler manifold, $E$ is a holomorphic vector bundle on $M$, and $L$ is a positive (complex) line bundle on $M$. Then, Serre's vanishing theorem says that when $n$ is sufficiently large, the $q$-th ($q>0$) cohomology $H^q(M, E\otimes L^n)$ of the sheaf of sections of $E\otimes L^n$ is equal to $0$. Here, $L^n$ is the tensor product of $n$ copies of $L$. 

We expect that in certain cohomology theory that shares common features (e.g., the $(p,q)$-diamond) with the sheaf cohomology or with the Dolbeault cohomology, there may exist such a Serre type vanishing property. Several candidates are the primitive cohomology and the $p$-filtered cohomology \cite{tanaka_tseng_2018, tty3rd, tty1st, tty2nd} for symplectic manifolds, and the generalized Dolbeault cohomology \cite{generalized_dolbeault_cohomology} for almost complex manifolds. 

In this paper, we obtain a Serre type vanishing property for the primitive cohomology twisted by a symplectically flat vector bundle. The background information on the symplectic flatness and the associated mapping cone covariant derivative is in \cite{tseng_and_zhou_symplectic_flat_connection_and_twisted_primitive2022, tseng_zhou_symplectic_flat_functional_characteristic_classes2022,  tseng_and_zhou_2025mapping_yang_mills, transgression_primitive_2025}. Throughout the paper, we use the following assumptions: 
\begin{assumption}\normalfont
    We let $(M,\omega)$ be a compact symplectic manifold. All vector bundles are real and smooth. For any vector bundle $E$ on $M$, we let $\Gamma(E)$ be the space of smooth sections of $E$, and $\Omega^i(M, E)$ be the space of smooth $E$-valued $i$-forms on $M$. 
\end{assumption}

Let $E$ be a vector bundle on $M$ admitting a symplectically flat connection 
\begin{align}
    \nabla^E: \Gamma(E)\to\Omega^1(M,E). 
\end{align}
Following \cite[Definition 1.1]{tseng_and_zhou_symplectic_flat_connection_and_twisted_primitive2022}, this means there exists $\Phi^E\in\Gamma(\edmp(E))$ such that 
\begin{align}
    (\nabla^E)^2 + \omega\wedge\Phi^E =\ & 0, \label{the first flatness}\\
    \Phi^E\nabla^E =\ & \nabla^E\Phi^E.  \label{the second flatness}
\end{align}
By \cite[Theorem 4.11]{tseng_and_zhou_symplectic_flat_connection_and_twisted_primitive2022}, the associated mapping cone covariant derivative 
\begin{align}
    \mathbb{A}^E: \Omega^i(M,E)\oplus\Omega^{i-1}(M,E)&\to \Omega^{i+1}(M,E)\oplus\Omega^{i}(M,E)   \nonumber \\
       (\alpha, \beta) &\mapsto (\nabla^E\alpha+\omega\wedge\beta, \Phi^E\alpha -\nabla^E\beta)
\end{align}
gives a cochain complex ($0\leqslant i\leqslant \dim M + 1$), whose cohomology is the primitive cohomology of $(M,\omega)$ twisted by $(E, \nabla^E, \Phi^E)$. 

Next, we let $(L, \nabla^L, \Phi^L)$ be a symplectically flat line bundle over $M$. Then, we let $L^n$ be the tensor product of $n$ copies of $L$, $\nabla^{E\otimes L^n}$ be the tensor product connection on $E\otimes L^n$ induced by $\nabla^E$ and $\nabla^L$, and $\Phi^{E\otimes L^n}$ be the derivation on $\Gamma(E\otimes L^n)$ induced by $\Phi^E$ and $\Phi^L$. They form the mapping cone covariant derivative
\begin{align}\label{tensor product cohomology}
\begin{split}
      \mathbb{A}^{E\otimes L^n}: \Omega^i(M,E\otimes L^n)\oplus\Omega^{i-1}(M,E\otimes L^n)&\to\Omega^{i+1}(M,E\otimes L^n)\oplus\Omega^{i}(M,E\otimes L^n)  \\
      (\alpha, \beta) &\mapsto (\nabla^{E\otimes L^n}\alpha+\omega\wedge\beta, \Phi^{E\otimes L^n}\alpha -\nabla^{E\otimes L^n}\beta).
      \end{split}
    \end{align}
Since $(E\otimes L^n, \nabla^{E\otimes L^n}, \Phi^{E\otimes L^n})$ is symplectically flat (see Proposition \ref{prerequisite}), (\ref{tensor product cohomology}) is a cochain complex $(0\leqslant i\leqslant \dim M + 1)$.

Finally, our main result, the vanishing property is as follows.
\begin{theorem}\label{main result}
   If $\Phi^L$ is everywhere non-zero, then for all sufficiently large $n$, the cohomology of the cochain complex {\normalfont(\ref{tensor product cohomology})} vanishes. 
\end{theorem}

\begin{remark}\normalfont
    We can replace $\omega$ by any closed $2$-form $\eta$ and prove Theorem \ref{main result} under the cone flatness with respect to $\eta$ \cite[(1.11)]{tseng_and_zhou_2025mapping_yang_mills}. The details are at the end of the paper. 
\end{remark}

We clarify the details of (\ref{tensor product cohomology}) in Section \ref{section of product of symplectically flat connections} and prove Theorem \ref{main result} in Section \ref{proof section}. 

\section*{Acknowledgments}
I thank \professor Xiaobo Liu, \professor Xiang Tang, \professor Li-Sheng Tseng, \professor Shanwen Wang, \doctor Jinxuan Chen, \doctor Shitan Xu, and \doctor Zheng Xu for helpful discussions. Also, I thank Beijing International Center for Mathematical Research for providing an excellent working environment.

\section{Verifying the symplectic flatness}\label{section of product of symplectically flat connections}
Before we prove Theorem \ref{main result}, we clarify the construction of (\ref{tensor product cohomology}) and show the symplectic flatness. 

Let $(E, \nabla^E, \Phi^E)$ and $(F, \nabla^F, \Phi^F)$ be two symplectically flat vector bundles on $M$.  We have the derivation $\Phi^{E\otimes F}\in\Gamma(\edmp(E\otimes F))$ determined by 
\begin{align}
    \Phi^{E\otimes F}(a\otimes b) = \Phi^{E}(a)\otimes b + a\otimes\Phi^{F}(b)
\end{align}
for all $a\in\Gamma(E)$ and $b\in\Gamma(F)$. Also, we have the tensor product connection \cite[Section 1.1]{bgv} given by 
\begin{align}
    \nabla^{E\otimes F}_X(a\otimes b) = (\nabla^{E}_Xa)\otimes b + a\otimes (\nabla^{F}_Xb)
\end{align}
for all $a\in\Gamma(E)$, $b\in\Gamma(F)$, and $X\in\Gamma(TM)$. 

\begin{remark}
\normalfont
    The derivation $\Phi^{E\otimes F}$ is actually $\Phi^E\otimes\text{id}_F+\text{id}_E\otimes\Phi^F$. 
\end{remark}
\begin{proposition}\label{prerequisite}
    The mapping cone covariant derivative 
    \begin{align}
    \mathbb{A}^{E\otimes F}: \Omega^i(M, E\otimes F)\oplus\Omega^{i-1}(M, E\otimes F)&\to \Omega^{i+1}(M, E\otimes F)\oplus\Omega^{i}(M, E\otimes F)   \nonumber \\
       (\alpha, \beta) &\mapsto (\nabla^{E\otimes F}\alpha+\omega\wedge\beta, \Phi^{E\otimes F}\alpha -\nabla^{E\otimes F}\beta)
    \end{align}
    is symplectically flat and thus defines a cochain complex $(0\leqslant i\leqslant \dim M + 1)$. 
\end{proposition}
\begin{proof}
    We only need to verify for the local tensor product expressions of sections. Let $\sigma\in\Omega^i(M)$, $v\in\Gamma(E)$, $w\in\Gamma(F)$. Then, we find
    \begin{align}
         \nabla^{E\otimes F}(\sigma\otimes(v\otimes w))  
        = d\sigma\otimes(v\otimes w) + (-1)^i\sigma\wedge(\nabla^E v\otimes w + v\otimes \nabla^F w),
    \end{align}
    and then
    \begin{align}
        & (\nabla^{E\otimes F})^2(\sigma\otimes(v\otimes w))    \nonumber\\
        =\ & \nabla^{E\otimes F}\left(d\sigma\otimes(v\otimes w) + (-1)^i\sigma\wedge(\nabla^E v\otimes w) + (-1)^i\sigma\wedge(v\otimes \nabla^F w)\right)   \nonumber\\
        =\ & (-1)^{i+1} d\sigma\wedge(\nabla^E v\otimes w) + (-1)^{i+1} d\sigma\wedge(v\otimes\nabla^F w)   \nonumber\\
        & + (-1)^i d\sigma\wedge(\nabla^E v\otimes w) + (-1)^{i+i}\sigma\wedge((\nabla^E)^2 v\otimes w) + (-1)^{i+i+1}\sigma\wedge(\nabla^E v\wedge\nabla^F w) \nonumber\\
        & + (-1)^i d\sigma\wedge(v\otimes \nabla^F w) + (-1)^{i+i}\sigma\wedge(\nabla^E v\wedge \nabla^F w) + (-1)^{i+i}\sigma\wedge(v\otimes(\nabla^F)^2 w) \nonumber\\
        =\ & \sigma\wedge((\nabla^E)^2 v \otimes w) + \sigma\wedge(v \otimes (\nabla^F)^2 w)   \nonumber\\
        =\ & -(\sigma\wedge\omega)\otimes(\Phi^E v\otimes w + v\otimes\Phi^F w)   \ \ \ \  (\text{By (\ref{the first flatness})})\nonumber\\
        =\ & -\omega\wedge\Phi^{E\otimes F}(\sigma\otimes(v\otimes w)).
    \end{align}
    In addition, 
    \begin{align}
        & \Phi^{E\otimes F}\nabla^{E\otimes F}(\sigma\otimes(v\otimes w))    \nonumber\\
        =\ & d\sigma\otimes(\Phi^E v\otimes w + v\otimes \Phi^F w) + (-1)^i\sigma\wedge(\Phi^E\nabla^E v\otimes w + \nabla^E v\otimes\Phi^F w)    \nonumber\\
        & + (-1)^i\sigma\wedge(\Phi^E v\otimes \nabla^F w + v\otimes \Phi^F\nabla^F w)    \nonumber\\
        =\ & d\sigma\otimes(\Phi^E v\otimes w + v\otimes \Phi^F w) + (-1)^i\sigma\wedge(\nabla^E\Phi^E v\otimes w + \nabla^E v\otimes\Phi^F w)    \nonumber\\
        & + (-1)^i\sigma\wedge(\Phi^E v\otimes \nabla^F w + v\otimes \nabla^F\Phi^F w)     \ \ \ \  (\text{By (\ref{the second flatness})})\nonumber\\
        =\ & d\sigma\otimes(\Phi^E v\otimes w + v\otimes \Phi^F w) + (-1)^i\sigma\wedge(\nabla^E\Phi^E v\otimes w + \Phi^E v\otimes \nabla^F w)    \nonumber\\
        & + (-1)^i\sigma\wedge(\nabla^E v\otimes\Phi^F w + v\otimes \nabla^F\Phi^F w)     \nonumber\\
        =\ & \nabla^{E\otimes F}\Phi^{E\otimes F}(\sigma\otimes(v\otimes w)).
    \end{align}
    Therefore, $\mathbb{A}^{E\otimes F}$ is symplectically flat if both $\mathbb{A}^E$ and $\mathbb{A}^F$ are symplectically flat. 
\end{proof}
\begin{remark}
\normalfont
    For $a\in\Gamma(E)$, $b\in\Gamma(F)$, and $\gamma, \delta\in\Omega^1(M)$, we define 
    \begin{align}\label{product of bundle-valued forms}
        (\gamma\otimes a)\wedge (\delta\otimes b)\coloneqq (\gamma\wedge\delta)\otimes(a\otimes b). 
    \end{align}
    This is a special case of \cite[Proposition 21.1]{loringtu2017differential} and makes sense of $\nabla^E v\wedge\nabla^F w$. If $E$ and $F$ are both $\mathbb{Z}_2$-graded superbundles (See \cite[Sections 1.3-1.4]{bgv} and \cite[Section 1.2]{weipingzhangnewedition}), we need the factor $(-1)^{\deg(\delta)\deg(a)}$ on the right hand side of (\ref{product of bundle-valued forms}). 
\end{remark}

\section{Proof of the main result}\label{proof section}
Now, we prove Theorem \ref{main result}. By Proposition \ref{prerequisite} and the vanishing property given in \cite[Theorem 4.2, Corollary 4.13]{tseng_and_zhou_symplectic_flat_connection_and_twisted_primitive2022}, we only need to show that $\Phi^{E\otimes L^n}$ is invertible when $n$ is sufficiently large. The main idea is in spirit very similar to the last step in the proof of Serre's vanishing theorem \cite[Proposition 5.2.7]{huybrechts2005complex}. 

Let $r$ be the rank of $E$. Then, we let $U$ and $V$ be two open subsets of $M$ satisfying $\overline{U}\subseteq V$,  
\begin{align}
    e_1, e_2, \cdots, e_r
\end{align}
be a frame of $E|_V$, and $v$ be a frame of $L|_V$. Then, we have smooth functions $f_{ij}$ ($1\leqslant i\leqslant r,\ 1\leqslant j\leqslant r$) and $f$ on $V$ such that 
\begin{align}\label{first matrix}
   \begin{bmatrix} \Phi^E(e_1), \cdots, \Phi^E(e_r) \end{bmatrix} = [e_1, \cdots, e_r]\begin{bmatrix}
       f_{11} & \cdots & f_{1r}\\
       \vdots &  & \vdots\\
       f_{r1} & \cdots & f_{rr}
   \end{bmatrix},
\end{align}
and 
\begin{align}\label{second matrix}
    \Phi^L(v) = f\cdot v.
\end{align}
For $1\leqslant i\leqslant r$, we let $w_i = e_i\otimes\underbrace{v\otimes\cdots\otimes v}_{n \text{\ copies}}$. Then, we have the frame
\begin{align}
    w_1, w_2, \cdots, w_r
\end{align}
of $E\otimes L^n$ on $V$. By (\ref{first matrix}) and (\ref{second matrix}), we have
\begin{align}
   & \begin{bmatrix} \Phi^{E\otimes L^n}(w_1), \Phi^{E\otimes L^n}(w_2), \cdots, \Phi^{E\otimes L^n}(w_r)\end{bmatrix}  \nonumber \\
   =\ & \begin{bmatrix}
       w_1, w_2, \cdots, w_r
   \end{bmatrix}\begin{bmatrix}
       f_{11} + n\cdot f & f_{12} & \cdots & f_{1r}\\
       f_{21} & f_{22} + n\cdot f & \cdots & f_{2r}\\
       \vdots & \vdots & \cdots & \vdots\\
       f_{r1} & f_{r2} & \cdots & f_{rr} + n\cdot f
   \end{bmatrix}. 
\end{align}
Since $\Phi^L$ is everywhere non-zero, and $M$ is compact, we have constants $C_1>0$ and $C_2>0$ such that on the compact subset $\overline{U}$ of $M$, 
\begin{align}
    |f|\geqslant C_1\text{\ and\ } |f_{ij}|\leqslant C_2\  (1\leqslant i\leqslant r,\ 1\leqslant j\leqslant r).
\end{align}
By the formula of determinant \cite[(4.16)]{greub1981linear}, for all sufficiently large $n$, the matrix
\begin{align}
    \begin{bmatrix}
       f_{11} + n\cdot f & f_{12} & \cdots & f_{1r}\\
       f_{21} & f_{22} + n\cdot f & \cdots & f_{2r}\\
       \vdots & \vdots & \cdots & \vdots\\
       f_{r1} & f_{r2} & \cdots & f_{rr} + n\cdot f
   \end{bmatrix}
\end{align}
is invertible on $\overline{U}$. The invertibility of $\Phi^{E\otimes L^n}$ is guaranteed by \cite[Lemma 2.3]{milnor1974characteristic} and the finite open cover consisting of such $U$. Theorem \ref{main result} is proved. 

We end the paper by explaining that $\omega$ can be replaced by any closed $2$-form $\eta$ on $M$. Our explanation is adapted from the proof of \cite[Theorem 4.2]{tseng_and_zhou_symplectic_flat_connection_and_twisted_primitive2022}. Given the cochain complex
\begin{align}\label{chain complex eta}
    \mathbb{A}^E: \Omega^i(M,E)\oplus\Omega^{i-1}(M,E)&\to \Omega^{i+1}(M,E)\oplus\Omega^{i}(M,E)   \nonumber \\
       (\alpha, \beta) &\mapsto (\nabla^E\alpha+\eta\wedge\beta, \Phi^E\alpha -\nabla^E\beta),
\end{align}
we assume that $\Phi^E$ is invertible, and $\nabla^E$ satisfies the cone flat conditions
\begin{align}
    (\nabla^E)^2 + \eta\wedge\Phi^E =\ & 0, \label{line 1}\\
    \Phi^E\nabla^E =\ & \nabla^E\Phi^E. \label{line 2}
\end{align}
Now, if $(\alpha, \beta)\in\Omega^i(M,E)\oplus\Omega^{i-1}(M,E)$ satisfies 
\begin{align}\label{line 3}
    (\nabla^E\alpha+\eta\wedge\beta, \Phi^E\alpha -\nabla^E\beta) = (0, 0), 
\end{align}
then by (\ref{line 1}), (\ref{line 2}), and (\ref{line 3}), we have
\begin{align}
    \mathbb{A}^{E}\left((\Phi^E)^{-1}\beta, 0\right) = (\alpha, \beta). 
\end{align}
Thus, when $\nabla^E$ is cone flat (with respect to $\eta$), and $\Phi^E$ is invertible, we see that the cohomology of (\ref{chain complex eta}) vanishes. Therefore, we can replace $\omega$ by $\eta$ in Theorem \ref{main result}. 

\bibliographystyle{abbrv}
\bibliography{mybib.bib}
\end{document}